\documentclass[11pt]{article}
\usepackage{fullpage,url,natbib}
\usepackage{amsthm,amsfonts,amsmath,amssymb,epsfig,color,float,graphicx,verbatim}
\usepackage{hyperref}
\hypersetup{
	colorlinks   = true, %Colours links instead of ugly boxes
	urlcolor     = blue, %Colour for external hyperlinks
	linkcolor    = blue, %Colour of internal links
	citecolor   = blue %Colour of citations
}
\newtheorem{proposition}{Proposition}
\newtheorem*{proposition*}{Proposition}

\renewcommand{\eqref}[1]{Eq.~(\ref{#1})}

\newcommand{\pred}[1]{\boldsymbol{1}[#1]}

\newcommand{\paren}[1]{\left( #1 \right)}

\newcommand{\vertiii}[1]{{\left\vert\kern-0.25ex\left\vert\kern-0.25ex\left\vert #1 
    \right\vert\kern-0.25ex\right\vert\kern-0.25ex\right\vert}}

\newcommand{\mexp}{\mathbb{E}}

\newcommand{\E}{\mathop{\mexp}}
\renewcommand{\P}{\mathbb{P}}

\newcommand{\eps}{\varepsilon}

\usepackage{xcolor}

\newcommand{\beq}{\begin{eqnarray*}}
\newcommand{\eeq}{\end{eqnarray*}}
\newcommand{\beqn}{\begin{eqnarray}}
\newcommand{\eeqn}{\end{eqnarray}}

\usepackage{ifmtarg}
\usepackage{xifthen}%
\newcommand{\ent}[1][]{%
\ifthenelse{\isempty{#1}}{%
\mathrm{H}
}{
\mathrm{H}^{(#1)}
}}

\newcommand{\loch}[1][]{%
\ifthenelse{\isempty{#1}}{%
\mathrm{h}
}{
\mathrm{h}^{(#1)}
}}

\newcommand{\mathe}{\mathrm{e}}
\newcommand{\mathd}{\mathrm{d}}

\newcommand{\hide}[1]{}

\newcommand{\set}[1]{\left\{ #1 \right\}}

\newtheorem*{rep@theorem}{\rep@title}
\newcommand{\newreptheorem}[2]{%
\newenvironment{rep#1}[1]{%
 \def\rep@title{#2 \ref{##1}}%
 \begin{rep@theorem}}%
 {\end{rep@theorem}}}
\makeatother

\newreptheorem{proposition}{Proposition}

\title{Decoupling Maximal Inequalities
}

\author{%
  Aryeh Kontorovich \\
  \texttt{karyeh@cs.bgu.ac.il}
  }

\date{}

\begin{document}

\maketitle
\begin{center}\vspace{-1cm}\today\vspace{0.5cm}\end{center}

\begin{abstract}
A {\em maximal inequality}
seeks to estimate $\E\max_i X_i$
in terms of properties of the $X_i$.
When the latter are independent,
the union bound (in its various guises)
can
yield tight upper bounds.
If, however, the $X_i$ are strongly dependent,
the estimates provided by the union bound will be rather loose.
In this note, we show that for non-negative random variables,
pairwise independence suffices for the maximal inequality
to behave comparably to its independent version.
The condition of pairwise independence may be relaxed to
a kind of negative dependence,
and even the latter admits violations --- provided
these are properly quantified.
\end{abstract}

\setcounter{section}{-1}
\section{Prolegomenon}
The key contributions of this note
were published as part of \citet{BlanchardCK24}
and so this note will not be submitted for peer review.
The result I attributed to Pinelis
(Proposition \ref{prop:pinelis})
had been obtained by \citet{PenaLai01}
with the constant $2$
and later improved by \citet{CHOLLETE202351}
to the constant given here.

\section{Motivation}
Maximal inequalities are at the heart of empirical process theory \citep{van2014probability}.
The case of Gaussian processes is well-understood via the celebrated generic chaining technique
\citep{Talagrand2016upper}. There, a key role in the lower
bounds is played Slepian's inequality,
which allows one to approximate a Gaussian process
by an appropriate uncorrelated one.
The absence of a generic analog of Slepian's inequality
--- say, for the kind of Binomal process considered in
\citet{CK22}
--- can be a major obstruction in obtaining tight lower
bounds. Indeed, as Proposition~\ref{prop:pinelis-cont}
below shows, 
for nonnegative $X_i$,
any upper bound on 
$\E\max_i \tilde X_i$,
where $\tilde X_i$
is the ``the independent version'' of $X_i$,
automatically yields an upper bound on
$\E\max_i X_i$.
The reverse direction, of course, fails without
additional structural assumptions. We discover that
pairwise independence suffices for the reverse direction,
and that this condition can be relaxed further.

\section{The Bernoulli case}

Let $X_1,X_2,\ldots,X_n$ 
and
$\tilde X_1,\tilde X_2,\ldots,\tilde X_n$
be two
%be a 
collections
of 
%(possibly dependent) 
%pairwise independent
Bernoulli random variables,
where 
%the $X_i$s are pairwise independent
%and 
the $\tilde X_i$s are 
mutually
independent (and independent of the $X_i$s),
with $X_i,\tilde X_i\sim\mathrm{Bernoulli}(p_i)$.
Letting
$Z=\sum_{i=1}^n X_i$
and
$\tilde Z=\sum_{i=1}^n \tilde X_i$,
we have
\beq
\E\max_{i\in[n]}X_i=\P(Z>0),
\qquad
\E\max_{i\in[n]}\tilde X_i=\P(\tilde Z>0).
\eeq
%Additionally, throughout this section, we will write
%$p_i:=\E[X_i]=\E[\tilde X_i]$

\paragraph{Decoupling from above.}
An elegant result of \citet{pinelis22}
(answering our question)
shows that
$\P(Z>0)\lesssim\P(\tilde Z>0)$;
his proof provided for completeness:
\begin{proposition}[Pinelis]
\label{prop:pinelis}
For $c=\mathe/(\mathe-1)$
and $X_i,\tilde X_i,Z,\tilde Z,p_i$ as above,
we have
\beq
\P(Z>0)\le c\P(\tilde Z>0).
\eeq
\end{proposition}
\begin{proof}
Put $M=\P(Z>0)$,
$\tilde M=\P(\tilde Z>0)$
and $S=\sum_{i=1}^n p_i$,
and observe that
\beq
M\le\min\set{S,1}\le c(1-\mathe^{-S}).
\eeq
On the other hand,
\beq
\tilde M=1-\prod_{i=1}^n(1-p_i)
\ge1-\mathe^{-S},
\eeq
whence $M\le c\tilde M$.
\end{proof}

Further, we note that Pinelis's constant
$c=\mathe/(\mathe-1)$
is optimal. Indeed,
consider the case where
$p_i=1/n$, $i\in[n]$, and $\P(Z=1)=1$.
This makes $\P(\tilde Z>0)=1-(1-1/n)^n\to 1-1/\mathe$
as $n\to\infty$.

Despite its elegance,
Proposition~\ref{prop:pinelis}
will likely have limited applications,
since in practice, the techniques for upper-bounding
$\E\max_i X_i$ rely on the union bound and
are insensitive to the dependence structure of $X_i$
--- in which case
the technique employed in upper-bounding
$\E\max_i X_i$ automatically upper-bounds
$\E\max_i \tilde X_i$ as well.

\paragraph{Decoupling from below.}
A more interesting and useful direction
would be to obtain an estimate of the form
$\P(Z>0)\gtrsim\P(\tilde Z>0)$.
Clearly, no such 
dimension-free
estimate can hold without
further assumptions on the $X_i$.
Indeed, 
for a small $\eps>0$,
let
$\P(X_1=X_2=\ldots=X_n=1)=\eps$
and
$\P(X_1=X_2=\ldots=X_n=0)=1-\eps$.
In this case, $\P(Z>0)=\eps$.
On the other hand,
$\P(\tilde Z>0)=1-(1-\eps)^n
=n\eps+O(\eps^2)
$, and so 
$
\P(\tilde Z>0)
/
\P(Z>0)
\to n
$
as $\eps\to0$.
Nor can the ratio exceed $n$,
%The ratio value $n$ is the largest possible, 
since
\beq
\E\max_{i\in[n]}\tilde X_i
\le
\sum_{i=1}^n\E \tilde X_i
\le
n\max_{i\in[n]}\E \tilde X_i
=
n\max_{i\in[n]}\E X_i
\le
n\E\max_{i\in[n]} X_i.
\eeq

Let us recall the notion of
pairwise independence. For the Bernoulli case,
it means that
for each $i\neq j\in [n]$, we have
$\E[X_iX_j]=
\E[X_i]\E[X_j]
$. The main result of this note
is that
pairwise independence suffices
for 
$\P(Z>0)\gtrsim\P(\tilde Z>0)$.

\begin{proposition}
\label{prop:main}
Let $X_i,\tilde X_i,Z,\tilde Z,p_i$ 
be as
above,
and assume additionally that the $X_i$
are pairwise independent.
Then
\beq
\P(Z>0) \ge \frac12\P(\tilde Z>0).
\eeq
\end{proposition}
\begin{proof}
By the Paley-Zygmund inequality,\footnote{
We thank Ron Peled for the suggestion of applying
Paley-Zygmund to $Z$.
}
\beq
\P(Z>0) \ge
\frac{(\E Z)^2}{\E[Z^2]}.
\eeq
Now $\E Z=\sum_{i=1}^n p_i$
and, by pairwise independence,
\beqn
\label{eq:Z^2}
\E[Z^2]
=
\sum_{i=1}^n p_i
+
2\sum_{1\le i<j\le n} p_ip_j
=
\sum_{i=1}^n p_i
+
\paren{\sum_{i=1}^n p_i}^2
-
\sum_{i=1}^n p_i^2
\le
\sum_{i=1}^n p_i
+
\paren{\sum_{i=1}^n p_i}^2
.
\eeqn
Hence,
\beq
\frac{(\E Z)^2}{\E[Z^2]}
\ge
\frac{
\paren{\sum_{i=1}^n p_i}^2
}{
\sum_{i=1}^n p_i
+
\paren{\sum_{i=1}^n p_i}^2
}
.
\eeq
On the other hand,
$\P(\tilde Z>0)$ is readily computed:
\beq
\P(\tilde Z>0)
=
1-\prod_{i=1}^n(1-p_i).
\eeq
Therefore, to prove the claim, it suffices to show that
\beq
F(p_1,\ldots,p_n)
:=
2\paren{\sum_{i=1}^n p_i}^2
-
\paren{\sum_{i=1}^n p_i
+
\paren{\sum_{i=1}^n p_i}^2
}
\paren{1-\prod_{i=1}^n(1-p_i)}
\ge0.
\eeq
To this end,\footnote{
This elegant proof that $F\ge0$
is due to D. Berend, who also corrected
a mistake in an earlier, clunkier proof of ours.} we factorize 
$F=SG$,
where
$G=S+P+SP-1$,
$S=\sum_i p_i$
and $P=\prod_i(1- p_i)$.
\iffalse
%We begin by observing that
$F=
G
\sum_{i=1}^n p_i
$,
where
\beq
G(p_1,\ldots,p_n)
=
2\sum_{i=1}^n p_i
-
\paren{1
+
\sum_{i=1}^n p_i
}
\paren{1-\prod_{i=1}^n(1-p_i)}
.
\eeq
\fi
Thus,
%It follows that 
$F\ge0\iff G\ge 0$
and in particular, it suffices to
verify the latter.
%Opening the parentheses, we have
%$G=S+P+SP-1$,
%where $S=\sum_i p_i$
%and $P=\prod_i(1- p_i)$.
Now if $S\ge1$ then obviously $G\ge0$
and we are done.
Otherwise,
since $P\ge1-S$
trivially holds,
we have
$G\ge S(1-S)$.
In this case, $S<1\implies G\ge0$.
\end{proof}

We conjecture that the constant $\frac12$ in Proposition~\ref{prop:main} is not optimal.
For a fixed $n$, define the joint 
pairwise independent
distribution on
$(X_1,\ldots,X_n)$ 
---
conjecturally, an extremal
one for minimizing
$
\P(Z=0)/
\P(\tilde Z>0)
$
---
as follows: $p_i=1/(n-1)$, $i\in[n]$,
%and 
$\P(Z=0)=
\frac12-\frac{1}{2(n-1)}
$, and 
$\P(Z=2)=1-\P(Z=0)$.
%the rest of the mass is uniformly
%spread over the $n$-tuples where $Z=2$.
This makes 
$\P(\tilde Z>0)=1-(1-1/(n-1))^n\to 1-1/\mathe$
as $n\to\infty$.
If our conjecture is correct, the optimal
constant for the lower bound is $c'=\frac{\mathe}{2(\mathe-1)}$, or exactly half of Pinelis's constant.\footnote{
We thank Daniel Berend, Alexander Goldenshluger,
and
Yuval Peres
for raising the question of the constants.
AG 
(and also Omer Ben-Porat)
pointed out a possible connection to
{\em prophet inequalities}
--- 
and in particular, the
Bernoulli selection lemma
in \cite{Correa17}
and
\cite{Esfandiari17},
where some constants related to $c,c'$ appear.
It still appears that Propositions~\ref{prop:pinelis}
and~\ref{prop:main} do not trivially follow from known results.
}

\paragraph{Relaxing pairwise independence.}
An inspection of the proof
shows that we do not actually need
$\E[X_iX_j]=p_ip_j$,
but rather only
$\E[X_iX_j]\le p_ip_j$.
This condition is called {\em negative (pairwise) covariance} \citep{Dubhashi:1998:BBS:299633.299634}.

\section{Positive real case}
In this section, we assume that
$X_1,\ldots,X_n$ are nonnegative
integrable random variables
and the
$\tilde X_1,\ldots,\tilde X_n$
are their independent copies:
each $\tilde X_i$ is distributed identically to $X_i$
and the $\tilde X_i$ are mutually independent.

As a warmup, let us see how 
Proposition~\ref{prop:pinelis}
yields $
\E\max_{i\in[n]}X_i
\lesssim
\E\max_{i\in[n]}\tilde X_i
$:
\begin{proposition}
\label{prop:pinelis-cont}
Let
$X_1,\ldots,X_n$ be nonnegative and
integrable
with independent copies
$\tilde X_i$ as above.
For $c=\mathe/(\mathe-1)$, we have
\beq
\E\max_{i\in[n]}X_i
\le
c\E\max_{i\in[n]}\tilde X_i
.
\eeq
\end{proposition}
\begin{proof}
For $t>0$ and $i\in[n]$,
put 
$Y_i(t)=\pred{X_i>t}$,
$\tilde Y_i(t)=\pred{\tilde X_i>t}$
and
$Z(t)=\sum_{i=1}^nY_i(t)$,
$\tilde Z(t)=\sum_{i=1}^nY_i(t)$.
Then
\beq
\E\max_{i\in[n]}X_i
&=&
\int_0^\infty
\P\paren{
\max_{i\in[n]}X_i>t
}\mathd t
\\
&=&
\int_0^\infty
\P\paren{
Z(t)>0
}\mathd t
\\
&\le&
c\int_0^\infty
\P\paren{
\tilde Z(t)>0
}\mathd t
\\
&=&
c\int_0^\infty
\P\paren{
\max_{i\in[n]}\tilde X_i>t
}\mathd t
\\
&=&
c\E\max_{i\in[n]}\tilde X_i
.
\eeq
\end{proof}

For pairwise independent $X_i$,
we have a reverse inequality:
\begin{proposition}
\label{prop:main-cont}
Let
$X_1,\ldots,X_n$ be nonnegative and
integrable
with independent copies
$\tilde X_i$ as above.
If additionally the $X_i$
are pairwise independent, then
\beq
\E\max_{i\in[n]}X_i
\ge
\frac12\E\max_{i\in[n]}\tilde X_i
.
\eeq
\end{proposition}
\begin{proof}
    The proof is entirely analogous to
    that of Proposition~\ref{prop:pinelis-cont},
    except that Proposition~\ref{prop:main}
    is invoked in the inequality step.
\end{proof}

\paragraph{Relaxing pairwise independence.}
As before, the full strength of pairwise
independence of the $X_i$ is not needed.
The condition $\P(X_i>t,X_j>t)\le \P(X_i>t)\P(X_j>t)$
for all $i\neq j\in[n]$ and $t>0$ would suffice;
it is weaker than 
pairwise
{\em negative upper orthant dependence}
\citep{10.1214/aos/1176346079}.\footnote{
Thanks to Murat Kocaoglu
for this reference.
}

\section{Back to Bernoulli: beyond negative covariance}
What if the Bernoulli $X_i$
do not satisfy the negative
covariance condition
$\E[X_iX_j]\le p_ip_j$?
Proposition~\ref{prop:main} is not directly inapplicable,
but not all is lost.
For $i\neq j\in[n]$, define $\eta_{ij}$ by
\beq
\eta_{ij}=\paren{
\E[X_iX_j]
-
p_ip_j
}_+
\eeq
and put $\eta_{ii}:=0$.
Thus, 
%we always have 
$
\E[X_iX_j]\le p_ip_j+\eta_{ij}
$, and, repeating the calculation in \eqref{eq:Z^2},
\beq
\E[Z^2]
\hide{
=
\sum_{i=1}^n p_i
+
2\sum_{1\le i<j\le n} p_ip_j
=
\sum_{i=1}^n p_i
+
\paren{\sum_{i=1}^n p_i}^2
-
\sum_{i=1}^n p_i^2
}
\le
\sum_{i=1}^n p_i
+
\paren{\sum_{i=1}^n p_i}^2
+
\sum_{i,j\in[n]}\eta_{ij}
.
\eeq

Let us put 
$S=\sum_{i=1}^n p_i$,
$A=S^2$,
$B=
S
+
S^2
$,
$C=\frac12\P(\tilde Z>0)$,
and $H=\sum_{i,j\in[n]}\eta_{ij}$.
Now, for $A,B,C,H\ge 0$, we have
\beq
\frac{A}{B}\ge C
&\implies&
\frac{A}{B+H}\ge C\paren{1-\frac{H}{B+H}}.
\eeq
and 
so 
we obtain a generalization of
Proposition~\ref{prop:main}:
\begin{proposition}
\label{prop:eta}
Let $X_i,\tilde X_i,Z,\tilde Z,p_i,B,H$ 
be as
above.
Then
\beq
\P(Z>0) \ge \frac12\paren{1-\frac{H}{B+H}}\P(\tilde Z>0)
.
\eeq
\end{proposition}
When $H\lesssim B$, 
Proposition~\ref{prop:eta}
yields useful estimates.

\section*{Acknowledgments}
I thank Daniel Berend and Victor de la Pe\~na
for helpful discussions.

%\bibliographystyle{plainnat}
%\bibliography{refs}

\begin{thebibliography}{11}
\providecommand{\natexlab}[1]{#1}
\providecommand{\url}[1]{\texttt{#1}}
\expandafter\ifx\csname urlstyle\endcsname\relax
  \providecommand{\doi}[1]{doi: #1}\else
  \providecommand{\doi}{doi: \begingroup \urlstyle{rm}\Url}\fi

\bibitem[Blanchard et~al.(2024)Blanchard, Cohen, and
  Kontorovich]{BlanchardCK24}
Mo{\"{\i}}se Blanchard, Doron Cohen, and Aryeh Kontorovich.
\newblock Correlated binomial process.
\newblock In Shipra Agrawal and Aaron Roth, editors, \emph{The Thirty Seventh
  Annual Conference on Learning Theory, June 30 - July 3, 2023, Edmonton,
  Canada}, volume 247 of \emph{Proceedings of Machine Learning Research}, pages
  551--595. {PMLR}, 2024.
\newblock URL \url{https://proceedings.mlr.press/v247/blanchard24a.html}.

\bibitem[Chollete et~al.(2023)Chollete, {de la Pe\~na}, and
  Klass]{CHOLLETE202351}
Lor\'an Chollete, Victor {de la Pe\~na}, and Michael Klass.
\newblock The price of independence in a model with unknown dependence.
\newblock \emph{Mathematical Social Sciences}, 123:\penalty0 51--58, 2023.
\newblock ISSN 0165-4896.
\newblock \doi{https://doi.org/10.1016/j.mathsocsci.2023.02.008}.
\newblock URL
  \url{https://www.sciencedirect.com/science/article/pii/S0165489623000215}.

\bibitem[Cohen and Kontorovich(2023)]{CK22}
Doron Cohen and Aryeh Kontorovich.
\newblock Local glivenko-cantelli.
\newblock In Gergely Neu and Lorenzo Rosasco, editors, \emph{The Thirty Sixth
  Annual Conference on Learning Theory, {COLT} 2023, 12-15 July 2023,
  Bangalore, India}, volume 195 of \emph{Proceedings of Machine Learning
  Research}, page 715. {PMLR}, 2023.
\newblock URL \url{https://proceedings.mlr.press/v195/cohen23a.html}.

\bibitem[Correa et~al.(2017)Correa, Foncea, Hoeksma, Oosterwijk, and
  Vredeveld]{Correa17}
Jos\'{e} Correa, Patricio Foncea, Ruben Hoeksma, Tim Oosterwijk, and Tjark
  Vredeveld.
\newblock Posted price mechanisms for a random stream of customers.
\newblock In \emph{Proceedings of the 2017 ACM Conference on Economics and
  Computation} 2017
\newblock ISBN 9781450345279.
\newblock \doi{10.1145/3033274.3085137}.
\newblock URL \url{https://doi.org/10.1145/3033274.3085137}.

\bibitem[Dubhashi and Ranjan(1998)]{Dubhashi:1998:BBS:299633.299634}
Devdatt Dubhashi and Desh Ranjan.
\newblock Balls and bins: a study in negative dependence.
\newblock \emph{Random Struct. Algorithms}, 13\penalty0 (2):\penalty0 99--124,
  September 1998.
\newblock ISSN 1042-9832.
\newblock \doi{10.1002/(SICI)1098-2418(199809)13:2<99::AID-RSA1>3.0.CO;2-M}.
\newblock URL
  \url{http://dx.doi.org/10.1002/(SICI)1098-2418(199809)13:2<99::AID-RSA1>3.0.CO;2-M}.

\bibitem[Esfandiari et~al.(2017)Esfandiari, Hajiaghayi, Liaghat, and
  Monemizadeh]{Esfandiari17}
Hossein Esfandiari, MohammadTaghi Hajiaghayi, Vahid Liaghat, and Morteza
  Monemizadeh.
\newblock Prophet secretary.
\newblock \emph{SIAM Journal on Discrete Mathematics}, 31\penalty0
  (3):\penalty0 1685--1701, 2017.
\newblock \doi{10.1137/15M1029394}.
\newblock URL \url{https://doi.org/10.1137/15M1029394}.

\bibitem[Joag-Dev and Proschan(1983)]{10.1214/aos/1176346079}
Kumar Joag-Dev and Frank Proschan.
\newblock {Negative Association of Random Variables with Applications}.
\newblock \emph{The Annals of Statistics}, 11\penalty0 (1):\penalty0 286 --
  295, 1983.
\newblock \doi{10.1214/aos/1176346079}.
\newblock URL \url{https://doi.org/10.1214/aos/1176346079}.

\bibitem[Lai and de~la Pe\~na(2001)]{PenaLai01}
T.~L. Lai and Victor~H. de~la Pe\~na.
\newblock Theory and applications of decoupling.
\newblock In N.~Balakrishnan Ch.A.~Caralambides, Markos V.~Koutras, editor,
  \emph{Probability and Statistical Models with Applications}. Chapman \&
  Hall/CRC, 2001.

\bibitem[Pinelis(2022)]{pinelis22}
Iosif Pinelis.
\newblock Max decoupling inequality.
\newblock MathOverflow, 2022.
\newblock URL \url{https://mathoverflow.net/q/422636}.

\bibitem[Talagrand(2016)]{Talagrand2016upper}
Michel Talagrand.
\newblock \emph{Upper and lower bounds for stochastic processes}.
\newblock Springer, 2016.

\bibitem[van Handel(2014)]{van2014probability}
Ramon van Handel.
\newblock Probability in high dimension.
\newblock Technical report, PRINCETON UNIV NJ, 2014.

\end{thebibliography}

\end{document}